\documentclass[11pt]{article}
\usepackage{amsmath,amssymb,amsthm}
\usepackage{epsfig}
\usepackage{color}
\usepackage{enumerate}
\usepackage{algorithm}
\usepackage{algorithmic}
\usepackage{hyperref}

\hypersetup{colorlinks=true}

\hypersetup{colorlinks=true, linkcolor=blue, citecolor=blue,urlcolor=blue}

\title{Global offensive $k$-alliances in digraphs}

\author{Doost Ali Mojdeh$^1$, Babak Samadi$^2$ and Ismael G. Yero$^3$\\[0.5cm]
Department of Mathematics, University of Mazandaran,\\ Babolsar, Iran$^{1,2}$\\
{\it damojdeh@umz.ac.ir$^1$, } {\it samadibabak62@gmail.com$^2$}\\[0.2cm]
Departamento de Matem\'{a}ticas, Universidad de C\'{a}diz,\\ Algeciras, Spain$^3$\\
{\it ismael.gonzalez@uca.es$^3$}}
\date{}

\newtheorem{theorem}{Theorem}[section]
\newtheorem{corollary}[theorem]{Corollary}
\newtheorem{lemma}[theorem]{Lemma}

\newtheorem{proposition}[theorem]{Proposition}

\theoremstyle{definition}
\newtheorem{definition}[theorem]{Definition}
\theoremstyle{remark}

\begin{document}

\maketitle

\begin{abstract}

\noindent \ \ \
In this paper, we initiate the study of global offensive $k$-alliances in digraphs. Given a digraph $D=(V(D),A(D))$, a global offensive $k$-alliance in a digraph $D$ is a subset $S\subseteq V(D)$ such that every vertex outside of $S$ has at least one in-neighbor from $S$ and also at least $k$ more in-neighbors from $S$ than from outside of $S$, by assuming $k$ is an integer lying between two minus the maximum in-degree of $D$ and the maximum in-degree of $D$. The global offensive $k$-alliance number $\gamma_{k}^{o}(D)$ is the minimum cardinality among all global offensive $k$-alliances in $D$. In this article we begin the study of the global offensive $k$-alliance number of digraphs. For instance, we  prove that finding the global offensive $k$-alliance number of digraphs $D$ is an NP-hard problem for any value $k\in \{2-\Delta^-(D),\dots,\Delta^-(D)\}$ and that it remains NP-complete even when restricted to bipartite digraphs when we consider the non-negative values of $k$ given in the interval above. Based on these facts, lower bounds on $\gamma_{k}^{o}(D)$ with characterizations of all digraphs attaining the bounds are given in this work. We also bound this parameter for bipartite digraphs from above. For the particular case $k=1$, an immediate result from the definition shows that $\gamma(D)\leq \gamma_{1}^{o}(D)$ for all digraphs $D$, in which $\gamma(D)$ stands for the domination number of $D$. We show that these two digraph parameters are the same for some infinite families of digraphs like rooted trees and contrafunctional digraphs. Moreover, we show that the difference between $\gamma_{1}^{o}(D)$ and $\gamma(D)$ can be arbitrary large for directed trees and connected functional digraphs.
\end{abstract}
\vspace{1mm}


\section{Introduction and preliminaries}

Throughout this paper, we consider $D=(V(D),A(D))$ as a finite digraph with vertex set $V(D)$ and arc set $A(D)$ with neither loops nor multiple arcs (although pairs of opposite arcs are allowed). Also, $G=(V(G),E(G))$ stands for a simple finite graph with vertex set $V(G)$ and edge set $E(G)$. We use \cite{bg} and \cite{we} as references for some very basic terminology and notation in digraphs and graphs, respectively, which are not explicitly defined here.

For any two vertices $u,v\in V(D)$, we write $(u,v)$ as the \emph{arc} with direction from $u$ to $v$, and say $u$ is \emph{adjacent to} $v$, or $v$ is \emph{adjacent from} $u$. Given a subset $S$ of vertices of $D$ and a vertex $v\in V(D)$, the {\em in-neighborhood} of $v$ from $S$ ({\em out-neighborhood} of $v$ to $S$) is $N_S^{-}(v)=\{u\in S\mid(u,v)\in A(D)\}$ ($N_S^{+}(v)=\{u\in S\mid(v,u)\in A(D)\}$). The \emph{in-degree} of $v$ from $S$ is $deg_S^-(v)=|N_S^{-}(v)|$ and the \emph{out-degree} of $v$ to $S$ is $deg_S^+(v)=|N_S^{+}(v)|$. Moreover, $N_S^{-}[v]=N_{S}^{-}(v)\cup\{v\}$ ($N_S^{+}[v]=N_{S}^{+}(v)\cup\{v\}$) is the {\em closed in-neighborhood} ({\em closed out-neighborhood}) of $v$ from (to) $S$. In particular, if $S=V(D)$, then we simply say (closed) (in or out)-neighborhood and (in or out)-degree of $v$, and write $N_D^{-}(v)$, $N_D^{+}(v)$, $N_D^{-}[v]$, $N_D^{+}[v]$, $deg_D^-(v)$ and $deg_D^+(v)$ instead of $N_{V(D)}^{-}(v)$, $N_{V(D)}^{+}(v)$, $N_{V(D)}^{-}[v]$, $N_{V(D)}^{+}[v]$, $deg_{V(D)}^-(v)$ and $deg_{V(D)}^+(v)$, respectively (we moreover remove the subscripts $D$, $V(D)$ if there is no ambiguity with respect to the digraph $D$). Given two sets $A$ and $B$ of vertices of $D$, by $(A,B)_D$ we mean the sets of arcs of $D$ going from $A$ to $B$. For a graph $G$, $\Delta=\Delta(G)$ and $\delta=\delta(G)$ represent the maximum and minimum degrees of $G$. In addition, for a digraph $D$, ($\Delta^{+}=\Delta^+(D)$ and $\delta^{+}=\delta^+(D)$) $\Delta^{-}=\Delta^-(D)$ and $\delta^{-}=\delta^-(D)$ represent the maximum and minimum (out-degrees) in-degrees of $D$.

We denote the {\em converse} of a digraph $D$ by $D^{-1}$, obtained by reversing the direction of every arc of $D$. A {\em biorientation} of a graph $G$ is a digraph $D$ which is obtained from $G$ by replacing each edge $xy$ by either $(x,y)$ or $(y,x)$ or the pair $(x,y)$ and $(y,x)$. While a {\em complete biorientation} $D$ of $G$ is obtained by replacing each edge $xy$ by the pair of arcs $(x,y)$ and $(y,x)$. A digraph $D$ is {\em connected} if its underlying graph is connected. A {\it directed tree} is a digraph in which its underlying graph is a tree. A {\em rooted tree} is a connected digraph with a vertex of in-degree $0$, called the {\em root}, such that every vertex different from the root has in-degree $1$. In general, we call a vertex with in-degree $0$ (out-degree $0$) in a digraph $D$ a {\em source} ({\em sink}). A digraph is {\em functional} ({\em contrafunctional}) if every vertex has out-degree (in-degree) $1$.

Given a graph $G$, a set $S\subseteq V(G)$ is a {\em dominating set} in $G$ if each vertex in $V(G)\setminus S$ is adjacent to a vertex in $S$. The {\em domination number} $\gamma(G)$ is the minimum cardinality of a dominating set in $G$. For more information about this concept the reader can consult \cite{hhs2}. The concept of domination in directed graphs was introduced by Fu \cite{fu}. A subset $S$ of the vertices of a digraph $D$ is called a {\em dominating set} if every vertex in $V(D)\setminus S$ is adjacent from a vertex in $S$. The {\em domination number} $\gamma(D)$ is the minimum cardinality of a dominating set in $D$.

Hedetniemi et al. \cite{hhk} introduced the concept of global offensive alliances in graphs. A subset $S\subseteq V(G)$ is said to be a {\em global offensive alliance} in $G$ if $|N[v]\cap S|\geq|N[v]\cap \overline{S}|$ for each $v\in \overline{S}$, where $\overline{S}$ is the complement of the set $S$ in $V(G)$. The {\em global offensive alliance number} $\gamma_{o}(G)$ is the minimum cardinality taken over all global offensive alliances in the graph $G$. As a generalization of such alliances, Shafique and Dutton \cite{sd1,sd2} defined the global offensive $k$-alliances in graphs. A set $S$ of vertices of a graph $G$ is called a {\em global offensive $k$-alliance} (GO$k$A for short) if $N[S]=V(G)$ and $|N(v)\cap S|\geq|N(v)\cap \overline{S}|+k$, for each $v\in \overline{S}$. The {\em global offensive $k$-alliance number} $\gamma_{k}^{o}(G)$ is the minimum cardinality of a GO$k$A in the graph $G$. For more information on global offensive ($k$-)alliances in graphs we suggest the surveys \cite{fernau-sur,alli-sur-other}.

Alliances in graphs have been a relatively popular research topic in graph theory in the last two decades, and a significant number of works dealing with them can be found through the literature. However, although the alliances are arising in a more natural way in a digraph than in a graph, the case of alliances in digraphs has not attracted the attention of any research till the recent work \cite{alli-def-dig}, where global defensive alliances in digraphs have been introduced. Consider now a social network (Twitter for instance) and an external entity which wants to spread some information in a positive sense, but that can be taken as false or as true by any user based on the number of opinions that can get from another users (if one receives more true opinions, it will take it as true, otherwise it will take it as false). Suppose that the entity gives the information to a set of users $S$ of the network. Hence, in order that the information will arrive in a true way to every user of the network, it is necessary that any other user $x\notin S$, that can hear the news from the elements in $S$, will have a larger number of connections inside the set $S$ than outside, otherwise, the information will be taken as false by $x$. Thinking in this way, it is readily observed that such a set $S$ must be a global offensive ($k$-)alliance in such network, that can be seen as such set of elements which are more influential among every one. For the sake of efficiency, the search of a minimum number of elements that can be used to spread such kind of information is then connected with precisely finding the global offensive ($k$-)alliance number of graphs. If such network uses directions in the connections (like in the case of the Twitter platform), then the definition of global offensive alliances in digraphs is clearly of interest for the study of these kinds of problems, and thus the following definition, and results concerning it are worthy.

\begin{definition}\label{Def1}
Let $D$ be a digraph and let $k\in\{2-\Delta^-(D), \dots , \Delta^-(D)\}$ be an integer. A set of vertices $S\subseteq V(D)$ is called a global offensive $k$-alliance (GO$k$A) in $D$ if $N^{+}[S]=V(D)$ and $deg_{S}^{-}(v)\geq deg_{\overline{S}}^{-}(v)+k$, for each $v\in \overline{S}$. The global offensive $k$-alliance number, denoted $\gamma_{k}^{o}(D)$, is defined as the minimum cardinality of a GO$k$A in $D$. We call the global offensive $1$-alliance (number) just {\em global offensive alliance \emph{(}number\emph{)}}, for short.
\end{definition}

In this paper, we first dedicate a section to the computational complexity of the problem of computing the global offensive $k$-alliance number of digraphs, by proving the NP-completeness of the respectively related decision problem. We next give several bounds on $\gamma_{k}^{o}(D)$ with some emphasis on the case $k=1$. For instance, we prove that $\gamma_{k}^{o}(D)$ can be bound from below by $(k+\delta^{-})n/(2\Delta^{+}+\delta^{-}+k)$ and characterize all digraphs $D$ attaining the lower bound for the specific case $k=1$. As a consequence of this result we improve a lower bound on $\gamma_{1}^o(G)=\gamma_{o}(G)$ (for graphs) given in \cite{sr}. Moreover, we show that $(n+n_{<k})/2$ is a sharp upper bound on $\gamma_{k}^{o}(D)$ for a bipartite digraph $D$, where $n_{<k}$ is the number of vertices of in-degree less than $k$. Also, we discuss some relationships between $\gamma_{1}^o(D)$ and $\gamma(D)$ with emphasis on (contra)functional digraphs and rooted trees.

From now on, given any parameter $\eta$ in a digraph $D$, a set of vertices of cardinality $\eta(D)$ is called an $\eta(D)$-set. Also, unless specifically stated, in the whole article we shall assume $k\in\{2-\Delta^-(D),\cdots,\Delta^-(D)\}$.


\section{Complexity issues}

One first basic observation with respect to the parameter above is the existent relationship between global offensive $k$-alliances of graphs and that of digraphs. Let $G$ be a graph and $D$ be a digraph obtained as a complete biorentation of $G$. We can immediately observe that a set of vertices $S$ is a global offensive $k$-alliance in $G$ if and only if $S$ is a global offensive $k$-alliance in $D$. This leads to the next result for which we omit its straightforward proof.

\begin{proposition}\label{graph-vs-digraph}
For any graph $G$ and any integer $k\in\{2-\Delta(G),\cdots,\Delta(G)\}$, $\gamma_{k}^{o}(G)=\gamma_{k}^{o}(D)$, where $D$ is the complete biorientation of $G$.
\end{proposition}

Such a relationship is very useful for giving a complexity result for the problem of computing the global offensive alliance number of digraphs. On the other hand, the result is less useful while studying general digraphs, since only digraphs for which an arc $(u,v)$ exists if and only if the arc $(v,u)$ also exists can be considered.

We now consider the problem of deciding whether the global offensive $k$-alliance number of a digraph is less than a given integer. That is stated in the following decision problem.

$$\begin{tabular}{|l|}
  \hline
  \mbox{GLOBAL OFFENSIVE $k$-ALLIANCE PROBLEM}\\
  \mbox{INSTANCE: A digraph $D$, an integer $k\in \{2-\Delta^-(D),\dots,\Delta^-(D)\}$, and a positive integer $r$}\\
  \mbox{PROBLEM: Deciding whether $\gamma_{k}^{o}(D)$ is less than $r$}\\
  \hline
\end{tabular}$$

Proving the NP-completeness of the GLOBAL OFFENSIVE $k$-ALLIANCE PROBLEM (GO$k$-A PROBLEM for short) above can be easily done (and therefore omitted) by making use of the Proposition \ref{graph-vs-digraph}, and the fact the the decision problem concerning computing the global offensive $k$-alliance number of graphs is NP-complete (see \cite{fernau-off-comp}).

\begin{theorem}\label{NP-complete}
For a digraph $D$ and an integer $k\in \{2-\Delta^-(D),\dots,\Delta^-(D)\}$, the GO$k$-A PROBLEM is NP-complete.
\end{theorem}

We now center our attention into bipartite digraphs, and prove that the GO$k$-A PROBLEM remains NP-complete, even when restricted to such class of digraphs if we consider $k\in \{0,\dots,\Delta^+(D)\}$. By a {\em bipartite digraph} we mean a biorientation of a bipartite graph (see \cite{bg}). In order to deal with this, we make a reduction from the well-known exact cover by $3$-sets problem (EC3S problem). That is, we have a set $A$ of exactly $n$ different elements, where $n$ is a multiple of three, and exactly $n$ subsets of $A$ such that every subset contains exactly $3$ elements of $A$ and every element occurs in exactly $3$ sets.
It can be readily seen that at least $\frac{n}{3}$ sets are needed to cover all the $n$ elements. In this sense, it is well-known that deciding whether there are $\frac{n}{3}$ such sets is in fact NP-complete (see~\cite{garey}).

\begin{theorem}\label{NP-complete}
For a digraph $D$ and an integer $k\in \{0,\dots,\Delta^-(D)\}$, the GO$k$-A PROBLEM is NP-complete for bipartite digraphs.
\end{theorem}

\begin{proof}
The problem is in NP, since for any given set of vertices $S$ of a digraph $D$, one can check in polynomial time that such set $S$ is indeed a global offensive $k$-alliance or not. We now describe a polynomial transformation of the EC3S problem to the GO$k$-A PROBLEM.

Consider a set of $A$ of exactly $n$ different elements, where $n$ is a multiple of three, and exactly $n$ subsets of $A$, such that every subset contains exactly $3$ elements of $A$ and every element occurs in exactly $3$ sets. Let $A=\{v_1,\dots,v_n\}$ and $U=\{U_1,\dots,U_n\}$ be the set of elements and the collection of subsets of elements of $A$, respectively. Let us construct a digraph $D$ as follows. For any element of $v_i\in A$ we create a vertex $v_i$ of $D$, and for any set of $U_i\in U$, we create a vertex $u_i$ of $D$. If an element $v_i$ occurs in a set $U_j$, then we add the arcs $(v_i,u_j)$ and $(u_j,v_i)$ (two opposite arcs). Now, for any vertex $v_i\in A$, we add $k+2$ vertices $v_{i,1},\dots,v_{i,k+2}$ and the arcs $(v_{i,1},v_i),\dots,(v_{i,k+2},v_i)$, and for any vertex $u_i$, we add $k+3$ vertices $u_{i,1},\dots,u_{i,k+3}$ and the arcs $(u_{i,1},u_i),\dots,(u_{i,k+3},u_i)$. We can easily note that the digraph constructed in this way is bipartite.

We shall now prove that deciding whether there are $\frac{n}{3}$ subsets in $U$ which cover the set $A$ is equivalent to prove that $D$ has global offensive $k$-alliance number equals to $\frac{n}{3}+n(2k+5)$.

We first assume that there are $\frac{n}{3}$ sets, without loss of generality say $U_1,\dots,U_{n/3}$, which cover the set $A$. Let $S$ be the set of vertices of $D$ given by the union of the sets $\{u_1,\dots,u_{n/3}\}$ and $\left\{v_{i,j},u_{l,q},\,:\,i,l\in \{1,\dots,n\},j\in\{1,\dots,k+2\},q\in\{1,\dots,k+3\}\right\}$. Note that any vertex $u_j$ with $j> n/3$ satisfies that $deg^-_S(u_j)=k+3=deg^-_{\overline{S}}(u_j)+k$. Moreover, since any vertex $v_i$ with $i\in\{1,\dots,n\}$ occurs in exactly one set $U_l$ with $l\in \{1,\dots,n/3\}$, it is satisfied that $deg^-_S(v_i)=k+3>k+2=deg^-_{\overline{S}}(v_i)+k$. Thus, $S$ is a GO$k$A in $D$, and so, $\gamma_{k}^{o}(D)\le \frac{n}{3}+n(2k+5)$.

On the other hand, let $S'$ be a $\gamma_k^o(D)$-set. Since any vertex $v_{i,j}$ and any vertex $u_{l,q}$ with $i,l\in \{1,\dots,n\}$, $j\in\{1,\dots,k+2\}$ and $q\in\{1,\dots,k+3\}$ has in-degree zero, we deduce that such vertices must belong to $S'$, which means
\begin{equation}\label{v-ij-u-lq}
|S'\cap \left\{v_{i,j},u_{l,q},\,:\,i,l\in \{1,\dots,n\},j\in\{1,\dots,k+2\},q\in\{1,\dots,k+3\}\right\}|=n(2k+5).
\end{equation}
Now, if there is a vertex $v_i$ for which $N^-_{S'}(v_i)\cap \{u_1,\dots,u_n\}=\emptyset$, then $deg^-_{S'}(v_i)=k+2<k+3=deg^-_{\overline{S'}}(v_i)+k$, which is not possible. Thus, any vertex $v_i$, with  $i\in\{1,\dots,n\}$ must have an in-neighbor in $S'\cap \{u_1,\dots,u_n\}$. Let $t=|S'\cap \{u_1,\dots,u_n\}|$. Since every vertex $v_{i}$ has at least one in-neighbor in $S'\cap \{u_1,\dots,u_n\}$ and every vertex $u_{i}$ has three out-neighbors in $\{v_1,\dots,v_n\}$, we have
\begin{equation}\label{U-s}
3t\geq \sum_{i=1}^{n}|N^{-}_{S'\cap \{u_1,\dots,u_n\}}|\geq n.
\end{equation}
Therefore, by using (\ref{v-ij-u-lq}) and (\ref{U-s}), we deduce that $\gamma_{k}^{o}(D)\ge \frac{n}{3}+n(2k+5)$, which leads to the desired equality.

We now assume that $\gamma_{k}^{o}(D)=\frac{n}{3}+n(2k+5)$ and let $Q$ be a $\gamma_k^o(D)$-set. As stated while proving the previous implication, it must happen that
$$\left\{v_{i,j},u_{l,q},\,:\,i,l\in \{1,\dots,n\},j\in\{1,\dots,k+2\},q\in\{1,\dots,k+3\}\right\}\subset Q.$$
Moreover, we can similarly see that $|Q\cap \{u_1,\dots,u_n\}|\ge n/3$, and that every vertex in the set $\{v_1,\dots,v_n\}$ has at least one in-neighbor in $Q\cap \{u_1,\dots,u_n\}$. Since $\gamma_{k}^{o}(D)=\frac{n}{3}+n(2k+5)$, it must happen that $|Q\cap \{u_1,\dots,u_n\}|=n/3$, which leads to that every vertex in $\{v_1,\dots,v_n\}$ has exactly one in-neighbor in $Q\cap \{u_1,\dots,u_n\}$. Let $W=Q\cap \{u_1,\dots,u_n\}$ (note that $|W|=n/3$). If the sets of $C$ (without loss of generality say $C_1,\dots,C_{n/3}$), corresponding to the vertices of $W$, do not form an exact cover of $U$, then either there is an element of $A$ which is not in any set $C_1,\dots,C_{n/3}$ or there is an element of $A$ which belongs to two sets of $C_1,\dots,C_{n/3}$. Both situations lead to a contradiction with the fact that $|W|=n/3$ and every vertex $v_{i}$, $i\in \{1,\dots,n\}$, has exactly one in-neighbor in $W$.
Therefore, $C_1,\dots,C_{n/3}$ form an exact cover of the elements in $A$, and this completes the proof of this implication, and the desired reduction.
\end{proof}

As a consequence of the two theorems above, we obtain that computing the global offensive $k$-alliance number of digraphs in an NP-hard problem for any suitable value of $k$, and it is moreover NP-hard even when restricted to bipartite digraphs.

\section{Bounding the global offensive $k$-alliance number}

Since the problem of computing the global offensive $k$-alliance number of digraphs is NP-hard, it is then desirable to bound it for general digraphs. We begin with exhibiting a lower bound on $\gamma_{k}^{o}(D)$ for a general digraph $D$. In order to characterize all digraphs attaining the bound with $k=1$, we define the family $\Phi$ of digraphs as follows.
Suppose that $\widehat{D}$ is a digraph with the set of vertices $\{v_{1},\cdots,v_{n'},u_{1},\cdots,u_{p}\}$ such that\vspace{1mm}\\
(i) $(r'+1)n'\equiv0$ (mod $p$) and $(r'+1)n'/p\geq deg^{+}_{\widehat{D}}(v_{i})$, for each $1\leq i\leq n'$,\vspace{1mm}\\
(ii) the in-degrees of all vertices $v_{i}$ in $\widehat{D}\langle\{v_{1},\cdots,v_{n'}\}\rangle$ equal $r'$,\vspace{1mm}\\
(iii) $deg^{+}_{\widehat{D}}(u_{i})=0$ and $deg^{-}_{\widehat{D}}(u_{i})\geq2r'+1$, for each $1\leq i\leq p$.\vspace{1mm}\\
We now add $r=(r'+1)n'/p$ arcs from each $u_{i}$, $1\leq i\leq p$, to the vertices in $\{v_{1},\cdots,v_{n'}\}$ such that all vertices $v_{i}$ are incident to $r'+1$ such arcs. Let $D$ be the obtained digraph.

As an example, let $D$ be obtained from the complete biorientation of the cycle $C_{t}$ on vertices $v_{1}\cdots,v_{t}$ with $t\geq5$, by adding three new vertices $u_{1}$, $u_{2}$ and $u_{3}$ and the set of new arcs $\{(u_{i},v_{1}),\cdots,(u_{i},v_{t})\}_{i=1}^{3}\cup \{(v_{j},u_{1}),(v_{j},u_{2}),(v_{j},u_{3})\}_{j=1}^{5}$. Then, $D$ is a member of $\Phi$ with $(n',p,r',r)=(t,3,2,t)$, in which $\widehat{D}$ is the graph with $V(\widehat{D})=V(D)$ and $E(\widehat{D})=E(D)\setminus \{(u_{i},v_{1}),\cdots,(u_{i},v_{t})\}_{i=1}^{3}$.

\begin{theorem}\label{T2}
If $D$ is a digraph of order $n$, minimum in-degree $\delta^-$ and maximum in-degree $\Delta^-$, then
$$\gamma_{k}^{o}(D)\geq\left(\frac{k+\delta^{-}}{2\Delta^{+}+\delta^{-}+k}\right)n.$$
Moreover, for the case $k=1$, the equality in the bound follows if and only if $D\in \Phi$.
\end{theorem}

\begin{proof}
Let $S$ be a $\gamma_{k}^{o}(D)$-set. We have
\begin{equation}\label{EQ12}
\begin{array}{lcl}
\Delta^{+}|S|&\geq&|(S,\overline{S})_{D}|=\sum_{v\in\overline{S}}deg_{S}^{-}(v)\geq \sum_{v\in\overline{S}}(deg_{\overline{S}}^{-}(v)+k)\\[0.3cm]
&=&k|\overline{S}|+\sum_{v\in\overline{S}}deg^{-}(v)-\sum_{v\in\overline{S}}deg_{S}^{-}(v)\geq(k+\delta^{-})|\overline{S}|-\Delta^{+}|S|.
\end{array}
\end{equation}
Therefore, the bound can be deduced from the above. We next consider the case $k=1$, which particularly means $\gamma_{1}^{o}(D)\geq\left(\frac{1+\delta^{-}}{2\Delta^{+}+\delta^{-}+1}\right)n$, and present the characterization of the digraphs achieving the equality in this situation.

Suppose that the lower bound holds with the equality for a digraph $D$. Hence, all the inequalities in (\ref{EQ12}) necessarily hold with equality. In particular, this means $\sum_{v\in\overline{S}}deg^{-}(v)=\delta^{-}|\overline{S}|$, which is equivalent to say that the in-degrees $deg^{-}(v)=deg_{S}^{-}(v)+deg_{\overline{S}}^{-}(v)$ of all vertices in $v\in V(D')$ equal $\delta^{-}$, where $D'$ is the subdigraph induced by $\overline{S}$. Moreover, $deg_{S}^{-}(v)=deg_{\overline{S}}^{-}(v)+1$ (note that $k=1$) for all $v\in V(D')$, by equality in the second inequality in (\ref{EQ12}). Therefore, all vertices in $V(D')$ have the same in-degree, say $r'$, in the subgraph induced by $V(D')$. On the other hand, every vertex in $S$ is adjacent to precisely $\Delta^{+}$ vertices of $D'$ since $\Delta^{+}|S|=|(S,\overline{S})_{D}|$. Now, since $deg_{S}^{-}(v)=deg_{\overline{S}}^{-}(v)+1=r'+1$ for all $v\in V(D')$, and $\Delta^{+}|S|=|(S,\overline{S})_{D}|$, we have that $|V(D')|(r'+1)=|S|\Delta^{+}$. Thus, the membership of $D$ in $\Phi$ easily follows by choosing $|V(D')|$, $D-(S,\overline{S})_{D}$, $\Delta^{+}$ and $S$ for $n'$, $\widehat{D}$, $r$ and the set $\{u_{1},\cdots,u_{p}\}$, respectively, in the description of $\Phi$. Thus, $D\in \Phi$.

Conversely, Let $D\in \Phi$. It can be observed that $\{u_{1},\dots,u_{p}\}$ is a GO$1$A in $D$. Moreover, $(n,\delta^{-},\Delta^{+})=(n'+p,2r'+1,n'(r'+1)/p)$. Thus, $\gamma_{1}^{o}(D)\leq p=(1+\delta^{-})n/(2\Delta^{+}+\delta^{-}+1)$. This completes the proof.
\end{proof}

As a result of the lower bound in Theorem \ref{T2} we have the following.

\begin{corollary}\label{Cor1} For any graphs $G$ of order $n$, minimum degree $\delta$ and maximum degree $\Delta$, $\gamma_{k}^o(G)\geq\lceil(k+\delta)n/(2\Delta+\delta+k)\rceil$.
\end{corollary}

\begin{proof}
Let $D$ be the complete biorientation of $G$. It is then straightforward to note that $|V(D)|=n$, $\delta^{+}(D)=\delta^{-}(D)=\delta$, $\Delta^{+}(D)=\Delta^{-}(D)=\Delta$, and $\gamma_{k}^{o}(D)=\gamma_k^{o}(G)$. Now the result follows from Theorem \ref{T2}.
\end{proof}

For the particular case of $k=1$, Sigarreta and Rodr\'{i}guez-Vel\'azquez \cite{sr} proved that
\begin{equation}\label{EQ4}
\gamma_1^{o}(G)\geq\left\{\begin{array}{lll}
\lceil(1+\delta)n/(2\Delta+\delta+1)\rceil,&\mbox{if $\delta$ is odd,}\vspace{1.5mm}\\
\lceil n\delta/(2\Delta+\delta)\rceil,&\mbox{otherwise.}
\end{array}
\right.
\end{equation}

Since $\lceil(1+\delta)n/(2\Delta+\delta+1)\rceil\geq\lceil n\delta/(2\Delta+\delta)\rceil$, Corollary \ref{Cor1} is an improvement of the lower bound given in (\ref{EQ4}).

We next continue with an upper bound on $\gamma_{k}^{o}(D)$.

\begin{theorem}\label{T3}
Let $D$ be a bipartite digraph of order $n$, let $k\in\{2-\Delta^-(D), . . . , \Delta^-(D)\}$, and let $n_{<k}$ be the number of vertices of in-degree less than $k$ in $D$. Then,
$$\gamma_{k}^{o}(D)\leq \frac{n+n_{<k}}{2},$$
and this bound is sharp.
\end{theorem}

\begin{proof}

We consider $V_{<k}$ as the set of all vertices of in-degree less than $k$. Let $X$ and $Y$ be the partite sets of $D$ and $X'=X\setminus V_{<k}$ and $Y'=Y\setminus V_{<k}$. Moreover, we may assume that $|X'|\geq|Y'|$. The above argument guarantees that each vertex in $X'$ has at least $k$ in-neighbors and all such in-neighbors belong to $Y$, necessarily. Therefore, $V(D)\setminus X'$ is a GO$k$A in $D$. Therefore,
\begin{equation}\label{EQ1}
\gamma_{k}^{o}(D)\leq n-|X'|\leq n-\frac{|X'|+|Y'|}{2}=n-\frac{|X|+|Y|-n_{<k}}{2}=\frac{n+n_{<k}}{2}.
\end{equation}

The sharpness of the upper bound can be seen as follows. We begin with the complete biorientation $D'$ of the complete bipartite graph $K_{p,p}$ with partite sets $X=\{x_{1},\cdots,x_{p}\}$ and $Y=\{y_{1},\cdots,y_{p}\}$ such that $k\leq p\leq 2k-1$. We obtain the digraph $D$ by removing the set of arcs $\{(x_{i},y_{j}),(y_{j},x_{i})\}_{i,j=1}^{k}$. This implies that $n_{<k}=2k$. Now, let $S$ be a $\gamma_{k}^{o}(D)$-set. Clearly, $\{x_{i},y_{i}\}_{i=1}^{k}\subseteq S$. If a partite set, say $X$, is a subset of $S$, then $|S|=|X|+|Y\cap S|\geq|X|+k=(n+n_{<k})/2$ which implies the equality in the upper bound. So, we may assume that both $X\setminus S$ and $Y\setminus S$ are nonempty. Let $t_{1}=|X\setminus S|$ and $t_{2}=|Y\setminus S|$. Suppose that $y\in Y\setminus S$. Since $deg^{-}_{S}(y)\geq deg^{-}_{\overline{S}}(y)+k$, we have $|X\cap S|\geq t_{1}+k$. Moreover, $|Y\cap S|\geq t_{2}+k$ by a similar fashion. Together the latest two inequalities imply $|S|\geq t_{1}+t_{2}+2k=n-|S|+n_{<k}$. Thus, $|S|\geq(n+n_{<k})/2$ which implies $\gamma_{k}^{o}(D)=(n+n_{<k})/2$ by (\ref{EQ1}). This completes the proof.
\end{proof}

One can observe that any directed tree is a bipartite digraph. So, as an immediate consequence of Theorem \ref{T3} we have the following result.

\begin{corollary}
Let $T$ be a directed tree of order $n$. Then, the following statements hold.\vspace{1mm}

\emph{(i)} $\gamma(T)\leq \lfloor(n+q)/2\rfloor$, where $q$ is the number of sources.\vspace{1mm}

\emph{(ii)} If $T$ is a rooted tree, then $\gamma(T)\leq \lceil n/2\rceil$. \emph{(\cite{l1})}
\end{corollary}


\subsection{The specific case of (contra)functional digraphs and rooted trees whether $k=1$}

As an immediate consequence of the definitions given in the introduction, we have $\gamma(D)\leq \gamma_{1}^{o}(D)$, for any digraph $D$. In this section, we investigate these two digraph parameters for (contra)functional digraphs and rooted trees.\vspace{2mm}

\begin{proposition}\label{P4}
For any rooted tree or contrafunctional digraph $D$, $\gamma_{1}^{o}(D)=\gamma(D)$.
\end{proposition}

\begin{proof}
Since every vertex in a contrafunctional digraph has in-degree one, every dominating set is a GO$1$A. Similarly, every dominating set is a GO$1$A in a rooted tree (the root of a rooted tree belongs to every dominating set and every GO$1$A). Therefore, $\gamma_{1}^{o}(D)\leq\gamma(D)$ which implies the equality.
\end{proof}

A vertex $y$ is {\em accessible} or {\em reachable} from $x$ if there is a path in $D$ from $x$ to $y$. Let $R(x)$ be the set of all vertices accessible from $x$ and let $R^{-1}(x)$ be the set of all vertices from which $x$ is accessible.

\begin{lemma}\emph{(\cite{h})}\label{L1}
A digraph $D$ is functional if and only if each of its components consists of exactly one directed cycle $C$ and for each vertex $v$ of $C$, the converse of subgraph induced by $R^{-1}(v)$ of the digraph $D-C$ is a rooted tree with the root $v$.
\end{lemma}

\begin{theorem}\label{T4}
Let $D$ be a connected functional digraph of order $n$ with $q$ sources. Then, $$\gamma_{1}^{o}(D)\leq\left\lfloor\frac{n+q+1}{2}\right\rfloor.$$
Furthermore, this bound is sharp.
\end{theorem}

\begin{proof}
We consider a connected functional digraph $D$ in view of Lemma \ref{L1}. Let $C$ be the unique directed cycle of $D$. Let $Q$ be the set of all sources in $D$. Then $D'=D-Q$ is still a connected functional digraph with the unique directed cycle $C$. We define the height $h(D')$ of the connected functional digraph $D'$ as $\max\{d_{D'}(v,V(C))\mid v\in V(D')\}$ where $d_{D'}(v,V(C))$ represents the length of a shortest path between $v$ and a vertex of $V(C)$. Let $D'_{1}=D'$. We select a source $v_{1}$ with maximum distance from $C$ and let $u_{1}$ be its unique out-neighbor. Let $D'_{2}=D'_{1}-N^{-}_{D_{1}}[u_{1}]$. Iterate this process for the remaining connected functional digraph $D_{i}$ until $D'_{p}$ is the directed cycle $C$ on vertices $w_{1},\dots,w_{|V(C)|}$ or a connected functional digraph with height one. In fact, we have a partition $\{Q,N^{-}_{D_{1}}[u_{1}],\dots,N^{-}_{D_{p-1}}[u_{p-1}],V(D'_{p})\}$ of $V(D)$. If $D'_{p}$ is the directed cycle $C$, then
$$S_{1}=Q\cup\{u_{i}\}_{i=1}^{p-1}\cup\{w_{2j-1}\}_{j=1}^{\lfloor(|V(C)|+1)/2\rfloor}$$
is a GO$1$A in $D$. Let $D'_{p}$ contain the directed cycle $C$ an some arcs $(w'_{j},w_{j})$ for some $1\leq j\leq|V(C)|$. We observe that $D'_{p}-\{w'_{j},w_{j}\}_{j}$ is either a disjoint union of some directed paths $P_{r}$ on vertices $x_{1},\dots,x_{r}$, or it is empty. If it is empty, then
$$S_{2}=Q\cup\{u_{i}\}_{i=1}^{p-1}\cup \{w_{j}\}_{j=1}^{|V(C)|}$$
is a GO$1$A in $D$. So, we assume that $D'_{p}-\{w'_{j},w_{j}\}_{j}$ is not empty. Let $V_{e}$ be the set of vertices on the directed paths $P_{r}$ with even subscripts. Then,
$$S_{3}=Q\cup\{u_{i}\}_{i=1}^{p-1}\cup\{w_{j}\}_{j}\cup V_{e}$$
is a GO$1$A in $D$.

On the other hand, $p-1\leq (n-q-|V(D'_{p})|)/2$. Moreover, the cardinalities of the sets $\{w_{2j-1}\}_{j=1}^{\lfloor(|V(C)|+1)/2\rfloor}$, $\{w_{j}\}_{j=1}^{|V(C)|}$ and $\{w_{j}\}_{j}\cup V_{e}$ are bounded from above by $(|V(D'_{p})|+1)/2$ for $S_{1}$, $S_{2}$ and $S_{3}$, respectively. Therefore, for any $i\in\{1,2,3\}$ we have,
$$\gamma_{1}^{o}(D)\leq|S_{i}|\leq q+(n-q-|V(D'_{p})|)/2+(|V(D'_{p})|+1)/2=(n+q+1)/2.$$

To see the sharpness of the bound, consider a directed cycle $C$ on vertices $y_{1},y_{2},\dots,y_{t}$ and add disjoint directed paths $x_{i,1},\dots,x_{i,2k_{i}+1}$ for each $1\leq i\leq t$, for which $x_{i,2k_{i}+1}$ is adjacent to $y_{i}$. Let $D^{*}$ be the obtained connected functional digraph. It is now easy to observe that $\{x_{i,1},x_{i,3},\dots,x_{i,2k_{i}+1}\}_{i=1}^{t}\cup\{y_{1},y_{3},\dots,y_{2\lfloor(t+1)/2\rfloor-1}\}$ is a minimum dominating set in $D^{*}$ of cardinality $\lfloor\frac{n+q+1}{2}\rfloor$. Therefore, $\gamma_{1}^{o}(D^{*})=\lfloor\frac{n+q+1}{2}\rfloor$.
\end{proof}

Note that the difference between $\gamma_{1}^{o}(D)$ and $\gamma(D)$ can be arbitrary large, even for connected functional digraphs and directed trees as we can see in the following examples.
Let $b$ be an arbitrary positive integer. Let $D'$ be obtained from a directed cycle $C$ on vertices $v_{1},\dots,v_{2b}$ by adding new vertices $v'_{1},\dots,v'_{2b}$ and arcs $(v'_{1},v_{1}),\dots,(v'_{2b},v_{2b})$. Then, $\{v'_{1},\dots,v'_{2b}\}$ is the minimum dominating set in $D'$ while $\{v'_{1},\dots,v'_{2b}\}\cup\{v_{2i}\}_{i=1}^{b}$ is a minimum GO$1$A in $D'$. Thus, $\gamma_{1}^{o}(D')-\gamma(D')=b$.
We now let $T$ be a directed tree by removing one arc from the directed cycle $C$ of $D'$. It is easy to see that $\gamma_{1}^{o}(T)-\gamma(T)=3b-2b=b$.



\section{Concluding remarks}

We have introduced and begun with the study of several combinatorial and computational properties of the global offensive $k$-alliances in digraphs. The results presented above have allowed us to generate a new research line on the theory of digraphs which we pretend to continue exploring by possibly dealing with some and/or all the following open problems.

\begin{itemize}
  \item Similarly to the case of graphs, alliances can be analyzed not only from a global way, but also in a local way. That is, for a given digraph $D=(V(D),A(D))$, one can consider a set of vertices $S\subseteq V(D)$ as an offensive $k$-alliance in $D$ if $deg^{-}_{S}(v)\geq deg^{-}_{\overline{S}}(v)+k$ for all $v\in N_D(S)\setminus S$ (which is equivalent to say that $S$ is not necessarily a dominating set in $D$). The offensive $k$-alliance number, which could be denoted $a_{k}^{o}(D)$, is then defined as the minimum cardinality of an offensive $k$-alliance in $D$. The study of the not global case for an offensive $k$-alliance is of potential interest to continue this research line, which we have presented in this article.
  \item Another issue that requires to be dealt with concerns completing the NP-hardness property of computing the global offensive $k$-alliance number of digraphs. That is, finding which is the complexity of the GO$k$-A PROBLEM studied above for the negative suitable values of $k$. It would probably be not surprising that such problem belongs to the so-called NP-hard class (as for the other values already proved here), however, a proof of it is required. In addition, finding some classes of digraphs whether such problem could be polynomially solved will give more insight into the study of global offensive $k$-alliances in digraphs.
  \item Since the global offensive $k$-alliances can be used to model the situation of finding the most influential elements of a network, it is worthy of finding some algorithms (that could be even not polynomial) together with some heuristics that would allow to make some implementations and experiments on real social networks in order to detect the ``influencers'' (to be according to the social network terminology) of such networks.
\end{itemize}


\end{document}